\documentclass{amsart}
\usepackage{graphicx}
\newtheorem{theorem}{Theorem}
\newtheorem{proposition}{Proposition}

\theoremstyle{definition}
\newtheorem{definition}{Definition}
\newtheorem{example}{Example}

\theoremstyle{remark}
\newtheorem{remark}{Remark}

\numberwithin{equation}{section}

\begin{document}

\title{A short proof of the converse to a theorem of Steinhaus}

\author{Dang Anh Tuan}
\address{Department of Mathematics-Mechanics-Informatics, Hanoi University of Science, VNU}
\email{danganhtuan@hus.edu.vn}


\date{August 03, 2015}          % Ngay


\begin{abstract}
A result of H. Steinhaus states that any positive Lebesgue measurable set has a property that its difference set contains an open interval around the origin. Y. V. Mospan proved that this result is the characterization of absolutely continuous measure. In this note we give a short proof of it.
\end{abstract}
\maketitle

Firstly, let me recall some facts from the paper "A converse to a theorem of Steinhaus" of Y. V. Mospan (\cite{M}).\\
\begin{definition}
A Borel measure $\mu$ on $\mathbb R$ is said to have Steinhaus property (abbreviated SP) if for every $\mu-$measurable set $A$ with $\mu(A)>0$ the difference $A-A$ contains an open interval around zero. 
\end{definition}
\begin{example}
H. Steinhaus stated that Lebesgue measure $m$ has SP property (\cite{S}).\\
It is easy to see Dirac measure $\delta$ and the Cantor measure do not have SP property. 
\end{example}
\begin{remark}
It is easy to see that a Borel measure $\mu$ is not SP if and only if there is Borel set $A$ with $\mu(A)>0$, and a sequence $\{t_n\}_{n=1}^\infty$ converging to $0$ such that
$$ A\cap (A+t_n)=\emptyset. $$ 
Note that "Borel set" can be replaced by "compact set".
\end{remark}
We give now some other characterizations of Borel measures which are not SP.
\begin{proposition}
Let $\mu$ be a Borel measure on $\mathbb R.$ Then the following statements are equivalent.
\begin{itemize}
\item[(i)] The Borel measure $\mu$ is not SP.
\item[(ii)] There is a compact set $A\subset \mathbb R$ such that $\mu(A)>0$ and the Lebesgue measure $m(A)$ of $A$ is zero.
\item[(iii)] There is a compact set $A\subset\mathbb R$ with $\mu(A)>0$ and a sequence $\{t_n\}_{n=1}^\infty$ converging to $0$ such that
$$ \lim\limits_{n\to\infty}\mu(A+t_n)=0. $$
\end{itemize}
\end{proposition}
\begin{proof}
We may assume that the support of $\mu$ is compact.\\
Firstly, we prove that (i) implies (ii). It is a easy consequence of Steinhaus Theorem. For the completion we give a short reason here. From the Remark 1 there is a compact set $A$ with $\mu(A)>0$ and a sequence $\{t_n\}_{n=1}^\infty$ converging to $0$ such that
$$ A\cap (A+t_n)=\emptyset. $$ 
So the Lebesgue measure $m(A\cap(A+t_n))=0$ for all $n\in\mathbb N.$ Note that
$$ m(A\cap (A+t_n))=\int_A\chi_{-A}(t_n-x)dm(x)=(\chi_A*\chi_{-A})(t_n),  m(A)=(\chi_A*\chi_{-A})(0)$$
and the map $t\mapsto (\chi_A*\chi_{-A})(t)$ is continuous from $\mathbb R$ to $\mathbb R.$  So (i) implies (ii).\\
Next, we show that (ii) implies (iii). We can consider $\mu$ and $m$ as  generalized functions 
$$ \mu: \varphi\in\mathcal D(\mathbb R)\mapsto \int\limits_{\mathbb R}\varphi(x)d\mu(x), $$
$$ m:  \varphi\in\mathcal D(\mathbb R)\mapsto \int\limits_{\mathbb R}\varphi(x)dm(x).$$
Then $\mu$ has compact support and $m$ is translation-invariant. Especially we have $m*\chi_{-A}$ is constant function $m(A)=0.$ Besides $\mu, \chi_{-A}$ has compact support so
$$m*(\mu*\chi_{-A})=\mu*(m*\chi_{-A})=0.  $$
Note that $(\mu*\chi_{-A})(x)=\mu(A+x)$ is a nonnegative function with compact support. Therefore $\mu*\chi_{-A}=0$ Lebesgue almost everywhere on $\mathbb R.$ Hence we get (iii).\\
Finally we show that (iii) implies (i). Let $A$ be a compact set satisfying (iii). Since
$$ \lim\limits_{n\to\infty}\mu(A+t_n)=0 $$
there is a subsequence $\{t_{n_k}\}_{k=1}^\infty$ such that
$$ \mu(A+t_{n_k})<\mu(A)/2^k. $$
Put $A'=A\setminus(\cup_{k=1}^\infty (A+t_{n_k}))$ we have
\begin{itemize}
\item $\mu(A')\ge \mu(A)-\sum\limits_{k=1}^\infty \mu(A+t_{n_k})>\mu(A)/2,$
\item $A'\cap (A'+t_{n_k})=\emptyset, \forall k\in\mathbb N,$
\item $\lim\limits_{k\to\infty}t_{n_k}=0.$
\end{itemize}
From the Remark 1 we get (i).
\end{proof}
From (ii) of the above Proposition we obtained the converse to theorem of Steinhauss.
\begin{theorem}
A Borel measure $\mu$ is SP if and only if $\mu$ is absolutely continuous to the Lebesgue measure $m.$
\end{theorem}
\begin{remark}
It is not difficult to have the same theorem for Borel measures on locally compact groups with Haar measure as the result of S. M. Simmons (\cite{Si}).
\end{remark}
\bibliographystyle{amsplain}

% Ket thuc van ban
\end{document}